\documentclass[a4paper,12pt,twoside, reqno]{amsart}
\usepackage{amssymb,amsthm,amsmath,bm,txfonts,graphicx,mathptmx}
\usepackage[all]{xy}

\makeatletter
  
  \@addtoreset{equation}{section}
\makeatother

\newtheorem{thm}{Theorem}[section]
\newtheorem{dfn}[thm]{Definition}

\newtheorem{prop}[thm]{Proposition}
\newtheorem{ex}[thm]{Example}
\newtheorem{cor}[thm]{Corollary}

\theoremstyle{remark}
\newtheorem{rem}[thm]{Remark}

\newcommand{\calN}{\mathcal{N}}

\newcommand{\fkm}{\mathfrak{m}}

\newcommand{\fkS}{\mathfrak{S}}

\newcommand{\NN}{\mathbb{N}}
\newcommand{\ZZ}{\mathbb{Z}}

\newcommand{\CC}{\mathbb{C}}
\newcommand{\FF}{\mathbb{F}}

\newcommand{\Spec}{\operatorname{Spec}}

\newcommand{\depth}{\operatorname{depth}}

\newcommand{\rank}{\operatorname{rank}}

\newcommand{\Hom}{\operatorname{Hom}}
\newcommand{\trace}{\operatorname{trace}}

\newcommand{\GL}{\operatorname{GL}}
\newcommand{\SL}{\operatorname{SL}}

\newcommand{\Sym}{\operatorname{Sym}}
\newcommand{\charac}{\operatorname{char}}

\newcommand{\Ref}{\operatorname{Ref}}

\begin{document}

\title[Generalized $F$-signature of invariant subrings]{Generalized $F$-signature of invariant subrings}
\author[Mitsuyasu Hashimoto \and YUSUKE NAKAJIMA]{Mitsuyasu Hashimoto \and YUSUKE NAKAJIMA}
\date{}

\subjclass[2010]{Primary 13A35, 13A50.}
\keywords{$F$-signature, finite $F$-representation type, invariant subrings}

\address[Mitsuyasu Hashimoto]{Department of Mathematics, Okayama University, 
Okayama 700-8530, Japan} 
\email{mh@okayama-u.ac.jp}
\address[Yusuke Nakajima]{Graduate School Of Mathematics, Nagoya University, Chikusa-Ku, Nagoya,
 464-8602, Japan} 
\email{m06022z@math.nagoya-u.ac.jp}
\maketitle

\begin{abstract}
It is known that a certain invariant subring $R$ has finite $F$-representation type. 
Thus, we can write the $R$-module ${}^eR$ as a finite direct sum of finitely many $R$-modules. 
In such a decomposition of ${}^eR$, we pay attention to the multiplicity of each direct summand. 
For the multiplicity of free direct summand, there is the notion of $F$-signature defined by C.~Huneke and G.~Leuschke 
and it characterizes some singularities. 
In this paper, we extend this notion to non free direct summands and determine their explicit values. 
\end{abstract}


\section{Introduction}

Throughout this paper, we suppose that $k$ is an algebraically closed
field of prime characteristic $p>0$, and $V$ is a $d$-dimensional
$k$-vector space.
Let $G\subset\GL(V)$ be a finite subgroup such that the order of
$G$ is not divisible by $p$, and $G$ contains no pseudo-reflections.
Let $S$ be a symmetric algebra of $V$. We denote the invariant subring of $S$ under the action of $G$ by $R\coloneqq S^G$. 
Sometimes we denote $p^e$ by $q$. 
Since $\operatorname{char}R=p>0$, we can define the Frobenius map $F:R\rightarrow R\;(r\mapsto r^p)$ and also 
define the $e$-times iterated Frobenius map $F^e:R\rightarrow R\;(r\mapsto r^{p^e})$ for $e\in\mathbb{N}$. 
For any $R$-module $M$, we denote the module $M$ with its $R$-module structure pulled back via $F^e$ by ${}^eM$.
That is, ${}^eM$ is just $M$ as an abelian group, and its $R$-module structure is given by $r\cdot m\coloneqq F^e(r)m=r^{p^e}m$ for all $r\in R,\;m\in M$. 

In our assumption, it is known that the invariant subring $R$ has finite $F$-representation type (or FFRT for short). 
The notion of FFRT is defined by K.~Smith and M.~Van den Bergh \cite{SVdB}.

\begin{dfn}
We say that $R$ has finite $F$-representation type by $\calN$ if 
there is a finite set $\calN$ of isomorphism classes of finitely generated $R$-modules,
such that for any $e\in\NN$, the $R$-module ${}^eR$ is isomorphic to a finite direct sum of elements of $\calN$.
\end{dfn}

More explicitly, finitely many finitely generated $R$-modules which form such a finite set $\calN$ are described as follows.  

\begin{prop} $($\cite[Proposition~3.2.1]{SVdB}$)$ 
Let $V_0=k,V_1,\cdots,V_n$ be a complete set of non-isomorphic irreducible representations of $G$ and we set $M_i\coloneqq(S\otimes_kV_i)^G\;\;(i=0,1,\cdots,n)$. 
Then $R$ has FFRT by $\{M_0\cong R,M_1,\cdots,M_n\}$.
\end{prop}

From this proposition, we can decompose ${}^eR$ as follows
\[
{}^eR\cong R^{\oplus c_{0,e}}\oplus M_1^{\oplus c_{1,e}}\oplus\cdots\oplus M_n^{\oplus c_{n,e}}.
\]

Now, we want to investigate the multiplicities $c_{i,e}$. For the multiplicity of the free direct summand, that is, for the multiplicity $c_{0,e}$, 
there is the notion of $F$-signature defined by C.~Huneke and G.~Leuschke.

\begin{dfn} $($\cite[Definition~9]{HL}$)$
The $F$-signature of $R$ is $s(R)\coloneqq\displaystyle\lim_{e\rightarrow\infty}\frac{c_{0,e}}{p^{de}}$, if it exists.
\end{dfn}

Note that K.~Tucker showed its existence under more general settings \cite[Theorem~4.9]{Tuc} $($see also Proposition\,\ref{unq_FFRT}$)$. 
And it is known that this numerical invariant characterizes some singularities. 
For example, $s(R)=1$ if and only if $R$ is regular \cite[Corollary~16]{HL} $($see also \cite[Theorem~3.1]{Yao2}$)$, 
and $s(R)>0$ if and only if $R$ is strongly $F$-regular \cite[Theorem~0.2]{AL}.
In our situation, the $F$-signature of the invariant subring $R$ is determined as follows and it implies that $R$ is strongly $F$-regular.

\begin{thm} $($\cite[Theorem~4.2]{WY}$)$ 
\label{Fsig_inv}
The $F$-signature of the invariant subring $R$ is 
\[
s(R)=\displaystyle\frac{1}{|G|}.
\]
\end{thm}

\begin{rem}
In \cite{WY}, this theorem is proved in terms of minimal relative Hilbert-Kunz multiplicity. 
And Y.~Yao showed that it coincides with the $F$-signature \cite[Remark~2.3 (4)]{Yao2}. 
\end{rem}

\medskip

Now, we extend this notion to other direct summands. Namely, we investigate the multiplicities $c_{i,e}\;(i=1,\cdots,n)$ 
and determine the limit $\displaystyle\lim_{e\rightarrow\infty}\frac{c_{i,e}}{p^{de}}$. 
In order to determine this limit, we have to care about the next two problems first.

\begin{itemize}
  \item For each $e\in\NN$, are the multiplicities $c_{i,e}$ determined uniquely?
  \item Does the limit $\displaystyle\lim_{e\rightarrow\infty}\frac{c_{i,e}}{p^{de}}$ exist?
\end{itemize}

In Section\,\ref{u.o.d}, we will show the uniqueness of the multiplicities. 
In Section\,\ref{GF-sig}, we will show the existence of the limit and determine the limit (see Theorem~\ref{main}).

\section{Uniqueness of decomposition}
\label{u.o.d}

In this section, we show the uniqueness of the multiplicities. Firstly, we introduce the notion of Frobenius twist $($e.g. \cite{Jan}$)$.

\begin{dfn}
For $k$-vector space $V$ and $e\in\ZZ$, we define $k$-vector space ${}^eV$ as follows
 \begin{itemize}
  \item ${}^eV$ is the same as $V$ as an additive group;
  \item the action of $\alpha\in k$ on ${}^eV$ is $\alpha\cdot v=\alpha^{p^e}v$.
 \end{itemize} 
An element $v\in V$, viewed as an element of ${}^eV$, is sometimes denoted by ${}^ev$. Thus $\alpha\cdot {}^ev={}^e(\alpha^{p^e}v)$.
By the composition 
$G\hookrightarrow\GL(V)\overset{\phi}{\rightarrow}\GL({}^eV)$, 
${}^eV$ is also a representation of $G$, where $\phi$ is given by $\phi(g)({}^ev)={}^e(gv)$ for $g\in G$ and $v\in V$.
We call this representation the Frobenius twist of $V$. Sometimes we denote this representation by $V^{(-e)}$.
\end{dfn}

Let $v_1,\cdots,v_d$ be a basis of $V$. For this basis, we suppose that a representation of $G$ is defined by
\[
g\cdot v_j=\sum^d_{i=1}f_{ij}(g)v_i \quad(g\in G,\;f_{ij}:G\rightarrow k).
\]
Namely, a matrix representation of $V$ is described by $\big(f_{ij}(g)\big)$. 
Since $k$ is an algebraically closed field, the basis $v_1,\cdots,v_d$ also forms a basis of ${}^eV$,
and the action of $G$ on ${}^eV$ is described as follows
\[
g\cdot {}^ev_j={}^e(g\cdot v_j)={}^e(\sum^d_{i=1}f_{ij}(g)v_i)=\sum^d_{i=1}f_{ij}(g)^{p^{-e}}({}^ev_i).
\]
From this observation, a matrix representation of the Frobenius twist ${}^eV$ is described by $\big((f_{ij}(g))^{p^{-e}}\big)$, 
that is, each component of the matrix representation of ${}^eV$ is the ${p^{-e}}$-th power of the original one. 

In order to show the uniqueness of the multiplicities, we prove the following. 

\begin{prop}
\label{unqness}
For $e\ge 1,\;c_{0,e},\cdots,c_{n,e}\ge0$, the following decompositions are equivalent
\begin{enumerate}
 \item[(1)] ${}^eR\cong M_0^{\oplus c_{0,e}}\oplus M_1^{\oplus c_{1,e}}\oplus\cdots\oplus M_n^{\oplus c_{n,e}}$ \quad as $R$-modules;
 \item[(2)] ${}^eS\cong(S\otimes_kV_0)^{\oplus c_{0,e}}\oplus(S\otimes_kV_1)^{\oplus c_{1,e}}\oplus\cdots\oplus(S\otimes_kV_n)^{\oplus c_{n,e}}$ \quad 
as $(G,S)$-modules;
 \item[(3)] ${}^eS/\fkm{}^eS\cong V_0^{\oplus c_{0,e}}\oplus V_1^{\oplus c_{1,e}}\oplus\cdots\oplus V_n^{\oplus c_{n,e}}$ \quad as $G$-modules;
 \item[(4)] there exist $\alpha_{ij}\in\frac{1}{q}\ZZ_{\ge0}$ such that ${}^eS\cong\displaystyle\bigoplus^n_{i=0}\bigoplus^{c_{i,e}}_{j=1}(S\otimes_kV_i)(-\alpha_{ij})$ \quad \\ 
       as $\frac{1}{q}\ZZ$-graded $(G,S)$-modules;
 \item[(5)] there exist $\alpha_{ij}\in\frac{1}{q}\ZZ_{\ge0}$ such that ${}^eR\cong\displaystyle\bigoplus^n_{i=0}\bigoplus^{c_{i,e}}_{j=1}M_i(-\alpha_{ij})$ \quad \\ 
       as $\frac{1}{q}\ZZ$-graded $R$-modules.
\end{enumerate}
\end{prop}

\begin{rem}
A similar correspondence holds for more general situation up to 
the action of the $e$-th Frobenius kernel of a group scheme \cite{Has}. 
For the case of a finite group $G$, the $e$-th Frobenius kernel of $G$ is trivial. Thus, we may ignore it in our context. 
\end{rem}

To prove this proposition, the next theorem plays the central role. 
This is proved in \cite{Has} for more general settings using some geometric settings.
For convenience of the reader, we give a short and simple proof using a result of Iyama and Takahashi
\cite{IT} or  Leuschke and Wiegand \cite{LW}.
The two-dimensional case is very well-known as a theorem of Auslander \cite{Auslander}.
See also \cite[Chapter~10]{Yo}.

\begin{thm}
\label{almost.p.f.b}
If $G$ contains no pseudo-reflections, then the functor $\Ref(G,S)\rightarrow\Ref(R)$ 
$(M\mapsto M^G)$ is an equivalence, where $\Ref(G,S)$ is the category of reflexive $(G,S)$-modules and 
$\Ref(R)$ is the category of reflexive $R$-modules. The quasi-inverse is $N\mapsto (S\otimes_RN)^{**}$.

The same functors give an equivalence ${}^*\Ref(G,S)\rightarrow{}^*\Ref(R)$, where ${}^*\Ref(G,S)$ is 
the category of $\ZZ[1/p]$-graded reflexive $(G,S)$-modules and ${}^*\Ref(R)$ is the category of $\ZZ[1/p]$-graded reflexive $R$-modules.
\end{thm}

\begin{proof}
Let $S*G$ denote the twisted group algebra.
It is $\bigoplus_{g\in G}S\cdot g$ as an $S$-module (with the free basis $G$),
and the multiplication is given by $(sg)(s'g')=(s(gs'))(gg')$.
A $(G,S)$-module and an $S*G$-module are one and the same thing.
As a $(G,S)$-module, $S*G$ and $S\otimes_k kG$ are the same thing, where
$kG$ is the group algebra (the left regular representation) of $G$ over $k$.
So $\Hom_S(S*G,S)\cong S\otimes_k k[G]\cong S\otimes_k kG\cong S*G$,
where $k[G]=(kG)^*$ is the $k$-dual of $kG$ (the left regular representation).

Let us denote by $S'$ the $R$-module $S$ with the trivial $G$-module structure.
Note that $S'\rightarrow (S\otimes_k k[G])^G$ given by
$s\mapsto \sum_{g\in G}gs\otimes e_g$ is an isomorphism,
where $\{e_g\mid g\in G\}$ is the dual basis of $k[G]$, dual to
$G$, which is a basis of $kG$.
Note that $g'e_g=e_{g'g}$.

For $M\in\Ref(G,S)$, $M^G$ is certainly reflexive.
Indeed, there is a presentation
\begin{equation}\label{presentation.eq}
(S*G)^u\rightarrow (S*G)^v\rightarrow M^*\rightarrow 0.
\end{equation}
Applying $(?)^G\circ \Hom_S(?,S)$, 
\[
0\rightarrow M^G\rightarrow (S')^v\rightarrow (S')^u
\]
is exact.
As it is easy to see that $S'$ satisfies the $(S_2)$-condition 
as an $R$-module
(that is, for $P\in\Spec R$, if $\depth_{R_P}(S'_P)<2$, then
$S'_P$ is a maximal Cohen--Macaulay $R_P$-module),
so is $M^G$, and it is reflexive.

On the other hand, it is obvious that $(S\otimes_R N)^{**}$ is
a reflexive $(G,S)$-module, since it is a dual of some $S$-finite
$(G,S)$-module.

Let $u:N\rightarrow ((S\otimes_R N)^{**})^G$ be the map given by
$u(n)=\lambda(1\otimes n)$, where $\lambda:S\otimes_R N\rightarrow
(S\otimes_R N)^{**}$ is the canonical map.
We show that $u$ is an isomorphism.
To verify this, since both $N$ and $((S\otimes_R N)^{**})^G$ are reflexive,
it suffices to show that 
\[
u_P:N_P\rightarrow (((S\otimes_R N)^{**})^G)_P\cong ((S_P\otimes_{R_P}N_P)^{**})
^G
\]
is an isomorphism for $P\in\Spec R$ with $\dim R_P\leq 1$ (cf. \cite[Lemma~5.11]{LW}). 
Then $N_P$ is a free module, and we may assume that $N_P=R_P$ by additivity.
This case is trivial.

Let $\varepsilon:(S\otimes_R M^G)^{**}\rightarrow M$ be the composite
\[
(S\otimes_R M^G)^{**}\xrightarrow{a^{**}} M^{**}\xrightarrow{\lambda^{-1}}M,
\]
where $a:S\otimes_R M^G\rightarrow M$ is given by $a(s\otimes m)=sm$.
We show that $\varepsilon$ is an isomorphism.
Since $(S\otimes_R M^G)^*$ and $M$ are reflexive, it suffices to show
that $a^*: M^*\rightarrow (S\otimes_R M^G)^*$ is an isomorphism.
By the five lemma and the existence of the presentation of the form
(\ref{presentation.eq}), we may assume that $M=S\otimes_k k[G]$.
Then $a^*$ is identified with the map
\[
S*G\cong (S\otimes_k k[G])^*\xrightarrow {a^*}(S\otimes_R 
(S\otimes_k k[G])^G)^*
\cong (S\otimes_R S')^*\cong \Hom_R(S',S).
\]
It is easy to see that this map is given by $sg\mapsto (s'\mapsto s(gs'))$.
This is an isomorphism by \cite[Theorem~4.2]{IT} or \cite[Theorem~5.12]{LW}.

As $u$ and $\varepsilon$ are isomorphisms,
$M\mapsto M^G$ and $N\mapsto (S\otimes_R N)^{**}$ are quasi-inverse to each other, and
hence they are category equivalences.

The graded version is proved similarly.
\end{proof}

By using this theorem, we give the proof of Proposition\,\ref{unqness}. 

\begin{proof}[Proof of Proposition\,\ref{unqness}]
The equivalences of $(1)$ and $(2)$, $(4)$ and $(5)$ follow from 
Theorem\,\ref{almost.p.f.b}, and $(3)$ is obtained by applying $(-\otimes_Sk)$ to $(2)$.
If we forget the grading from $(4)$, then we obtain $(2)$.

\[\begin{array}{rcclll}
  (1)&\overset{\mathrm{Thm}.\ref{almost.p.f.b}}{\Longleftrightarrow}&(2)&&\overset{\otimes_Sk}{\Longrightarrow}&(3) \\
     &&\Big\Uparrow&^{\text{forget}}_{\text{grading}}&&\\
     &&(4)&&\underset{\mathrm{Thm}.\ref{almost.p.f.b}}{\Longleftrightarrow}&(5)
\end{array}\]

So we will show $(3)\Rightarrow (4)$. 
If we consider ${}^eS/\fkm{}^eS$ as a $\frac{1}{q}\ZZ$-graded $G$-module, 
then we can write
\[
{}^eS/\fkm{}^eS\cong\bigoplus^n_{i=0}\bigoplus^{c_{i,e}}_{j=1}V_i(-\alpha_{ij})
\]
for some $\alpha_{ij}\in\frac{1}{q}\ZZ_{\geq 0}$.
Then 
as in the proof of 
\cite[Proposition~3.2.1]{SVdB},
we have 
${}^eS\cong S\otimes_k({}^eS/\fkm{}^eS)$, and (4) follows.
\end{proof}

Especially,
the decomposition $(3)$ appears in Proposition\,\ref{unqness} is unique. Thus, we obtain the next statement as a corollary. 

\begin{cor}
\label{cor_unq}
Each $M_i$ is indecomposable and the multiplicities $c_{i,e}$ are determined uniquely.
\end{cor}

In Proposition\,\ref{unqness} and Corollary\,\ref{cor_unq}, the condition ``$G$ contains no pseudo-reflections'' is essential.
If $G$ contains a pseudo-reflection, then there is a counter-example as follows.

\begin{ex}
Let $S=k[x,y]$ be a polynomial ring, where $(\charac\,k,|G|)=1$.
Set $G=\big<\sigma=
    \begin{pmatrix} 0&1 \\
                    1&0 
    \end{pmatrix}          \big>$, that is $G$ is a symmetric group $\fkS_2$, 
and, $V_0=k, V_1=\operatorname{sgn}$ are irreducible representations of $G$. (Note that $\sigma$ is a pseudo-reflection.) 
Then, $R\coloneqq S^G\cong k[x+y, xy]$. Since $R$ is a polynomial ring, ${}^eR\cong R^{\oplus p^{2e}}$.
On the other hand, 
\[
M_1\coloneqq (S\otimes_kV_1)^G=\{f\in S \mid \sigma\cdot f=(\operatorname{sgn}\;\sigma)f\}=(x-y)R\cong R.
\]
So ${}^eR$ also decompose as ${}^eR\cong M_1^{\oplus p^{2e}}$. Therefore, the uniqueness doesn't hold in this case. 
\end{ex}

\section{Generalized $F$-signature of invariant subrings}
\label{GF-sig}

In this section, we show the existence of the limit and determine it. 

In our case, the invariant subring $R$ has FFRT. 
Thus, the existence of the limit $\displaystyle\lim_{e\rightarrow\infty}\frac{c_{i,e}}{p^{de}}$ is guaranteed by the next proposition. So we can define this limit. 

\begin{prop} $($\cite[Proposition~3.3.1]{SVdB}, \cite[Theorem~3.11]{Yao}$)$ 
\label{unq_FFRT}
If $R$ has FFRT, then
for $i=0,1,\cdots,n$, the limit $\displaystyle\lim_{e\rightarrow\infty}\frac{c_{i,e}}{p^{de}}$ exists.
\end{prop}

\begin{rem}
In \cite{SVdB}, this proposition is proved under the assumption ``$R$ is strongly $F$-regular and has FFRT". 
After that, Y.~Yao showed the condition of strongly $F$-regular is unnecessary \cite{Yao}. 
Note that the existence of the limit for free direct summands $($i.e. $F$-signature$)$ is proved under more general settings as we showed before.
\end{rem}

\begin{dfn}
We call this limit the generalized $F$-signature of $M_i$ with respect to $R$ and
denote it by 
\[
s(R,M_i)\coloneqq\displaystyle\lim_{e\rightarrow\infty}\frac{c_{i,e}}{p^{de}}.
\]
\end{dfn}

The main theorem in this paper is the following. 

\begin{thm}[Main theorem]
\label{main}
Let the notation be as above. Then for all $i=0,\cdots,n$ one has
\[
 s(R,M_i)=\frac{\dim_kV_i}{|G|}=\frac{\rank_RM_i}{|G|}.
\]
\end{thm}

\begin{rem}
The second equation follows from $\dim_kV_i=\rank_RM_i$ clearly. 

The case that $i=0$ is due to \cite[Theorem~4.2]{WY}, as we have seen before $($Theorem\,\ref{Fsig_inv}$)$.
And a similar result holds for finite subgroup scheme of $\SL_2$ \cite[Lemma~4.10]{HS}. 
\end{rem}

\begin{rem}
From this theorem, we can see that each indecomposable MCM $R$-modules in the finite set $\{R, M_1,\cdots, M_n\}$ 
actually appear in ${}^eR$ as a direct summand for sufficiently large $e$ $($see also \cite[Proposition~2.5]{TY}$)$.
\end{rem}

\medskip

In order to prove this theorem, we introduce the notion of the Brauer character. 
In the representation theory of finite groups over $\CC$, the character gives us very effective method to distinguish each representation. 
But now, we are in a positive characteristic field $k$, not in $\CC$. So the character in the original sense doesn't work well. 
Therefore we have to modify it for applying to our context. 
For this purpose, we introduce the Brauer character $($for more details, refer to some textbooks e.g. \cite{CR}, \cite{Wei}$)$. 

\bigskip

As we assume that $m:=|G|$ is not divisible by $p$, there is a primitive $m$-th root of unity in $k$,
and thus both $\mu_m(k)=\{\omega\in k^\times\mid \omega^m=1\}$ and
$\mu_m(\CC)=\{\omega\in\CC^\times\mid \omega^m=1\}$ are the cyclic groups of 
order $m$.
Fix a group isomorphism $\Phi:\mu_m(k)\rightarrow \mu_m(\CC)$.

\begin{dfn}
For a $kG$-module $V$, the Brauer character $\chi_V$ of $V$ is the function $\chi_V: G\rightarrow \CC$ given by 
\[
 \chi_V(g)\coloneqq \sum^d_{i=1}\Phi(\omega_i) \in\CC \quad (g\in G), 
\]
where $\omega_1, \cdots, \omega_d$ are the eigenvalues of $g$. 
\end{dfn}

\medskip

The following proposition is well-known for the original character over $\CC$. 
And this kind of formula also holds for the Brauer character. 

\begin{prop}
\label{character prop}
Let $V,W$ be $kG$-modules and $g\in G$, then
\begin{enumerate}
  \item[(1)] $\chi_{V\otimes W}(g)=\chi_V(g)\cdot \chi_W(g)$.
  \item[(2)] $\chi_{V\oplus W}(g)=\chi_V(g)+\chi_W(g)$.
  \item[(3)] $\chi_{V^*}(g)=\overline{\chi_V(g)}$, where the bar denotes the conjugate of a complex number.
  \item[(4)] $\chi_V(1_G)=\dim_k V$.
  \item[(5)] $\dim_k V^G=\displaystyle\frac{1}{|G|}\sum_{g\in G}\chi_V(g)$.
  \item[(6)] $\dim_k\Hom_G(V,W)=\displaystyle\frac{1}{|G|}\sum_{g\in G}\overline{\chi_V(g)}\cdot\chi_W(g)$.
 \end{enumerate}
\end{prop}

\begin{proof}
The statements (1)--(4) follow easily from the definition.
(6) follows from (1), (3), and (5).
So we only prove (5).
If we show (5) for a particular choice of $\Phi$, then (5) is true for arbitrary choice, say $\Phi'$, because
we can write $\Phi'=\alpha\circ \Phi$, where $\alpha$ is some automorphism of $\mathbb Q(\omega)$ over $\mathbb Q$, 
where $\omega$ is a primitive $m$-th root of unity in $\CC$.
Let $R$ be the ring of Witt vectors over $k$.
Note that $R$ is a complete DVR (discrete valuation ring).
Let $t$ be its uniformizing parameter.
We identify $R/tR$ with $k$.
Let $\bar \omega$ be a fixed primitive $m$-th root of unity in $k$.
By Hensel's lemma, it is easy to see that $\bar \omega$ lifts to a primitive $m$-th root of unity in $R$ uniquely, say to $\omega$.
Note that $V$ is a $kG$-module, and hence is an $RG$-module.
Let $V_R\rightarrow V$ be the projective cover as an $RG$-module, which exists (note that $RG$ is semiperfect).
Note that $V_R/tV_R=V$, and $V_R$ is an $R$-free module of rank $\dim_k V$.

Let $R_0=\ZZ[\omega]$ be the subring of $R$ generated by $\omega$.
Then regarding 
$R_0$ as a subring of $\CC$, we have that $\tilde \chi_V$ is a Brauer character of $V$,
where $\tilde \chi_V(g)=\trace_{V_R}(g)$ (the trace makes sense, since $V_R$ is a finite free $R$-module).
Let $\gamma=\displaystyle\frac{1}{|G|}\sum_{g\in G}g\in RG$.
Then it is easy to see that $\gamma$ is a projector from any $RG$-module $M$ to $M^G$.
In particular, the $G$-invariance $(?)^G$ is an exact functor on the category of $RG$-modules.
It follows that $V^G=(V_R/tV_R)^G\cong V_R^G/tV_R^G=k\otimes_R V_R^G$.
Let $U:=(1-\gamma)V_R$.
Then $V_R=V_R^G\oplus U$, and $\gamma$ is the identity map on $V_R^G$ and zero on $U$.
So $\displaystyle\frac{1}{|G|}\sum_{g\in G}\tilde \chi_V(g)=\trace_{V_R}(\gamma)=\rank_R V_R^G=\dim_k V^G$.
This is what we wanted to prove.
\end{proof}

So we are now ready to prove the main theorem.

\begin{proof}[Proof of Theorem\,\ref{main}]
Firstly, there is $e_0\ge 1$ such that the group ring $\FF_{q_0}G$ is isomorphic to the direct product of total matrix rings over $\FF_{q_0}$, 
where $q_0=p^{e_0}$. Namely, 
\[
 \FF_{q_0}G\cong\operatorname{Mat}_{r_1}(\FF_{q_0})\times\cdots\times\operatorname{Mat}_{r_m}(\FF_{q_0}),\quad (r_1,\cdots,r_m\in\NN).
\]
Since the component of matrix representation of Frobenius twist is $p^{-e}$-th power of the original one, 
so if we take an appropriate basis, then any component of matrix representation is in the finite field $\FF_{q_0}$.   
Thus, if $e=e_0t$, then we can consider ${}^eM\cong M$ for any $G$-module $M$.

Since we know the existence of the limit, it suffices to show the subsequence $\{\frac{c_{i, e_0t}}{p^{de_0t}}\}_{t\in\NN}$ converges on 
$(\dim_kV_i)/|G|$. So we prove 
\[
\lim_{t\rightarrow\infty}\frac{c_{i, e_0t}}{p^{de_0t}}=\frac{\dim_kV_i}{|G|}.
\]

For $e=e_0t$, we obtain ${}^eS/\fkm{}^eS\cong {}^e(S/\fkm^{[q]})\cong S/\fkm^{[q]}$. 
And ${}^eS/\fkm{}^eS$ is also isomorphic to the finite direct sum of irreducible representations (cf.\;Proposition\,\ref{unqness}).
By Proposition\,\ref{character prop}\;(6), the multiplicity $c_{i,e}$ is described as follows.
\[
 c_{i,e}=\dim_k\Hom_G(V_i,S/\fkm^{[q]})=
 \frac{1}{|G|}\sum_{g\in G}\overline{\chi_{V_i}(g)}\cdot\chi_{S/\fkm^{[q]}}(g).
\]

Set $g\in G$ and suppose that the order of $g$ is $m$. 
Then there is a basis $\{x_1, \cdots, x_d\}$ of $V$ such that each $x_i$ is an eigenvector of $g$ and 
we can write $g\cdot x_i=\omega_i x_i$ with $\omega_i=\omega^{\delta_i}$ for some $0\leq \delta_i<m$, 
where $\omega$ is a primitive $m$-th root of unity. In this situation
\[
\{x_1^{\lambda_1}\cdots x_d^{\lambda_d}\mid 0\leq \lambda_1,\ldots,\lambda_d
<q\}\subset
\displaystyle\bigoplus_{l=0}^{(q-1)d}\Sym_l V
\]
is a basis of $S/\mathfrak m^{[q]}$. 
As each $x_1^{\lambda_1}\cdots x_d^{\lambda_d}$ is an eigenvector of $g$
with the eigenvalue $\omega_1^{\lambda_1}\cdots \omega_d^{\lambda_d}$, we have 
\[
\chi_{S/\mathfrak m^{[q]}}(g)
=
\sum_{0\leq \lambda_1,\ldots,\lambda_d<q}
\Phi(\omega_1^{\lambda_1}\cdots\omega_d^{\lambda_d})
=\prod_{i=1}^d(1+\theta_i+\cdots+\theta_i^{q-1}), 
\]
where $\theta_i\coloneqq\Phi(\omega_i)$.

\begin{enumerate}
  \item[(i)] In case $g=1$, by Proposition\,\ref{character prop}\,(4), 
  \[
   \frac{\overline{\chi_{V_i}(g)}\cdot\chi_{S/\fkm^{[q]}}(g)}{q^d}
  =\frac{\dim_kV_i\cdot q^d}{q^d}=\dim_kV_i.
  \]
 \end{enumerate}

\begin{enumerate}
  \item[(ii)] In case $g\neq 1$, we may assume $\theta_d\neq 1$. Then
  
   \[\Bigg|\; \frac{\overline{\chi_{V_i}(g)}\cdot\chi_{S/\fkm^{[q]}}(g)}{q^d}\;\Bigg|
   \le\; \frac{\big|\,\overline{\chi_{V_i}(g)}\,\big|}{q^d}\prod^{d-1}_{i=1}(|1|+|\theta_i|+\cdots+|\theta_i|^{q-1})\cdot\Bigg|   \frac{1-\theta_d^q}{1-\theta_d}\Bigg|\] \\
   \[\le\frac{\dim_kV_i}{q}\cdot\frac{2}{|1-\theta_d|}\xrightarrow{t\rightarrow\infty} 0. \]  
 \end{enumerate}
The first inequation is obtained by applying the triangle inequality. Since $|\theta_i|\le 1$, we can obtain the second inequation. 

From previous arguments, we may only discuss in case $g=1$. Thus, we conclude
\[
\lim_{e\rightarrow\infty}\frac{c_{i, e}}{q^d}
=\lim_{e\rightarrow\infty}\frac{1}{q^d}\cdot\frac{1}{|G|}\sum_{g\in G}\overline{\chi_{V_i}(g)}\cdot\chi_{S/\fkm^{[q]}}(g)
=\frac{\dim_kV_i}{|G|}.
\]
\end{proof}

   \bigskip
   
Next, we consider the decomposition of  ${}^eM_i$. Since each MCM $R$-module $M_i$ appears in ${}^{e^\prime} R$  
for sufficiently large $e^\prime \gg 0$ as a direct summand, 
it also decomposes as 
\[
{}^eM_i\cong M_0^{\oplus d_{0,e}^i}\oplus M_1^{\oplus d_{1,e}^i}\oplus\cdots\oplus M_n^{\oplus d_{n,e}^i}. 
\]
In this situation, we define the limit 
\[
s(M_i,M_j)\coloneqq\displaystyle\lim_{e\rightarrow\infty}\frac{d_{j,e}^i}{p^{de}}
\]  
and call it the generalized $F$-signature of $M_j$ with respect to $M_i$. 
The next corollary immediately follows from Theorem~\ref{main} and \cite[Proposition~3.3.1, Lemma~3.3.2]{SVdB}.   

\begin{cor}
Let the notation be as above. 
Then for all $i,j=0,\cdots,n$ one has 
\[
s(M_i,M_j)=(\dim_kV_i)\cdot s(R,M_j)=\frac{(\dim_kV_i)\cdot(\dim_kV_j)}{|G|}
=\frac{(\rank_RM_i)\cdot(\rank_RM_j)}{|G|}.
\]
\end{cor}

\bigskip

\subsection*{Acknowledgements}
The authors would like to thank Professor Kei-ichi~Watanabe for valuable advice and also thank Professor Peter~Symonds for 
reading the previous version of this manuscript and pointing out that the Brauer character is more suitable to prove the main theorem.


\end{document}